\documentclass[12pt]{amsart}
\usepackage[T2A]{fontenc}
\usepackage[utf8]{inputenc}   
\usepackage[russian, english]{babel}
\usepackage{enumerate}

\usepackage{url}
\usepackage{hyperref}
\usepackage{amsmath}
\usepackage{amsthm}
\usepackage{amsfonts}
\usepackage{amssymb}
\usepackage{cite}
\usepackage{bbm}
\usepackage{mathptmx} 
\usepackage[scaled=0.90]{helvet} 
\usepackage{courier} 
\usepackage[T1]{fontenc}
\usepackage{color}

\usepackage{graphicx}
\DeclareGraphicsExtensions{.eps}

\newtheorem{theorem}{Theorem}[section]

\newtheorem*{example*}{Example}
\newtheorem{lemma}[theorem]{Lemma}

\newtheorem*{th*}{Theorem}
\newtheorem{prop}[theorem]{Proposition}
\theoremstyle{definition}
\newtheorem*{zamech*}{Remark}
\newtheorem{defin}[theorem]{Definition}

\def\sli{\sum\limits}
\def\ili{\int\limits}

\def\R{\mathbb{R}}

\def\ep{\varepsilon}
\def\vf{\varphi}

\def\E{\mathbb{E}}

\newcommand{\supp}{\operatorname{supp}}

\newcommand{\diam}{\operatorname{diam}}

\newcommand{\PP}{\mathcal{P}}

\renewcommand{\P}{\mathbb{P}}

\def\cyr{\fontencoding{OT2}\fontfamily{wncyr}\selectfont}
\DeclareTextFontCommand{\textcyr}{\cyr}

{\end{list}}

\newcounter{vremennyj}

\def\H{\mathcal{H}}

\date{July 2016}
\author{A. Reznikov}
\address{Center for Constructive Approximation, Department of Mathematics, Vanderbilt University}
\email{aleksandr.b.reznikov@vanderbilt.edu}
\author{E. B. Saff}
\email{edward.b.saff@vanderbilt.edu}
\author{O. V. Vlasiuk}

\email{oleksandr.v.vlasiuk@vanderbilt.edu}
\thanks{The research of the authors was supported, in part, by the National Science Foundation grants DMS-1516400 and DMS-1412428}
\title[Polarization for integrable kernels]{A minimum principle for potentials with application to Chebyshev constants}
\begin{document}
\begin{abstract}
For ``Riesz-like'' kernels $K(x,y)=f(|x-y|)$ on $A\times A$, where $A$ is a compact $d$-regular set $A\subset \R^p$, we prove a minimum principle for potentials $U_K^\mu=\int K(x,y)\textup{d}\mu(x)$, where $\mu$ is a Borel measure supported on $A$. Setting $P_K(\mu)=\inf_{y\in A}U^\mu(y)$, the $K$-polarization of $\mu$, the principle is used to show that if $\{\nu_N\}$ is a sequence of measures on $A$ that converges in the weak-star sense to the measure $\nu$, then $P_K(\nu_N)\to P_K(\nu)$ as $N\to \infty$. The continuous Chebyshev (polarization) problem concerns maximizing $P_K(\mu)$ over all probability measures $\mu$ supported on $A$, while the $N$-point discrete Chebyshev problem maximizes $P_K(\mu)$ only over normalized counting measures for $N$-point multisets on $A$. We prove for such kernels and sets $A$, that if $\{\nu_N\}$ is a sequence of $N$-point measures solving the discrete problem, then every weak-star limit measure of $\nu_N$  as $N \to \infty$ is a solution to the continuous problem. 
\end{abstract}
\maketitle

\vspace{4mm}


\vspace{2mm}


\vspace{2mm}

\normalsize
\section{Introduction}
For a nonempty compact set $A\subset \R^p$, a kernel $K\colon A\times A\to \R\cup \{\infty\}$ and a measure $\mu$ supported on $A$, the $K$-potential of $\mu$ is defined by
$$
U_K^\mu(y):=\int_A K(x,y)\textup{d}\mu(x), \;\; y\in \R^p.
$$
Assuming that $K$ is lower semi-continuous, the Fatou lemma implies that if $y_n\to y$ as $n\to \infty$, we have
$$
\liminf_{n} U_K^{\mu}(y_n)\geqslant U_K^{\mu}(y);
$$
thus $U_K^{\mu}$ is a lower semi-continuous function on $\R^p$.
We define the weak$^*$ topology on the space of positive Borel measures as follows.
\begin{defin}
Let $(\mu_n)_{n=1}^\infty$ be a sequence of positive Borel measures supported on a compact set $A$. We say that the measures $\mu_n$ converge to the measure $\mu$ in the weak$^*$ sense, $\mu_n \stackrel{*}{\to} \mu$, if for any function $\vf$ continuous on $A$ we have
$$
\ili \vf(x)\textup{d}\mu_n(x)\to \ili \vf(x)\textup{d}\mu(x), \;\;\; n\to \infty.
$$
\end{defin}For a measure $\mu$ supported on $A$ its {\it $K$-polarization} is defined by
$$
P_K(\mu):=\inf_{y\in A} U_K^{\mu}(y).
$$
In the following definition we introduce two special constants which denote the maximum value of $P_K(\mu)$ when $\mu$ ranges over all probability measures and when $\mu$ ranges over all probability measures supported on finite sets.
\begin{defin}
For a positive integer $N$ the {\it discrete $N$-th $K$-polarization
(or Chebyshev) constant of $A$} is defined by
\begin{equation}\label{discrpolar}
\mathcal{P}_K(A, N):=\sup_{\omega_N\subset A} \inf_{y\in A} U_K^{\nu_{\omega_N}}(y),
\end{equation}
where the supremum is taken over {\it $N$-point multisets} $\omega_N$; i.e., $N$-point sets counting multiplicities, and where $\nu_{\omega_N}$ is the normalized counting measure of $\omega_N$:
$$
\nu_{\omega_N}:=\frac1N \sum_{x\in \omega_N} \delta_{x}.
$$

Moreover, we say that the probability measure $\nu$ supported on $A$ {\it solves the continuous $K$-polarization problem} if
$$
\inf_{y\in A}U^\nu_K(y) = \sup_{\mu}\inf_{y\in A}U_K^\mu (y)=:T_K(A),
$$
where the supremum is taken over all probability measures $\mu$ supported on $A$.
\end{defin}

The following result has been known since 1960's; it relates the asymptotic behavior of $\PP_K(A, N)$ as $N\to\infty$ with $T_K(A)$.
\begin{theorem}[Ohtsuka, \cite{Ohtsuka1967}]
Assume $A\subset \R^p$ is a compact set and $K\colon A\times A\to (-\infty, \infty]$ is a lower semi-continuous symmetric kernel bounded from below. Then
\begin{equation}\label{ohtsuka}
\PP_K(A, N)\to T_K(A), \; \; N\to \infty.
\end{equation}
\end{theorem}
What has been as yet unresolved for integrable kernels on sets $A$ of positive $K$-capacity is whether, under the mild assumptions of symmetry and lower semi-continuity of $K$, every limit measure (in the weak$^*$ sense) of a sequence of normalized counting measures $\nu_{\omega_N}$ associated with optimal $N$-th $K$-polarization constants attains $T_K(A)$.
We remark that such a result does not necessarily hold for non-integrable kernels. Consider a two-point set $A=\{0,1\}$ and any kernel $K$ with $K(0,0)=K(1,1)=\infty$ and $K(0,1)=K(1,0)<\infty$. Then, for any $N\geqslant 2$, the measure $\nu_N:=(1/N) \delta_0 + ((N-1)/N) \delta_1$ attains $\PP_K(A, N)=\infty$. However, $\nu_N \stackrel{*}{\to} \delta_1$, which does not attain $T_K(A)=\infty$.

One case when such a result holds is for $K\in C(A\times A)$. Namely, the following is true, see \cite{Borodachov2015}, \cite{Farkas2008}, \cite{Farkas2006} and \cite{Farkas2006a}.
\begin{theorem}
Let $A\subset \R^p$ be a compact set and $K\in C(A\times A)$ be a symmetric function. A sequence $(\omega_N)_{N=1}^\infty$ of $N$-point multisets on $A$ satisfies
$$
\lim_{N\to \infty} P_K(\nu_{\omega_N}) = T_K(A)
$$
if and only if every weak$^*$-limit measure $\nu^*$ of the sequence $(\nu_{\omega_N})_{N=1}^\infty$ attains $T_K(A)$.
\end{theorem}
Notice that this theorem does not cover cases when $K$ is unbounded along the diagonal of $A\times A$; in particular, Riesz kernels $K(x,y)=|x-y|^{-s}$ when $s>0$. The following theorem by B. Simanek applies to Riesz kernels (as well as more general kernels) but under rather special conditions on the set $A$ and the Riesz parameter.
\begin{theorem}[Simanek, \cite{Simanek2015}]\label{thsimanek}
Assume $A\subset \R^p$ is a compact set and $K\colon A\times A \to (-\infty, \infty]$ is a lower semi-continuous kernel bounded from below. Assume further there exists a unique probability measure $\mu_{eq}$ with $\supp(\mu_{eq})=A$ and $U^{\mu_{eq}}(y)\equiv C$ for every $y\in A$. Then $\mu_{eq}$ is the unique measure that attains $T_K(A)$. Furthermore, if $\nu_N$ is an $N$-point normalized counting measure that attains      $\PP_K(A, N)$, then  $\nu_N\stackrel{*}{\to} \mu_{eq}$ as $N\to \infty$.
 \end{theorem}

 We remark that if $A=\mathbb{B}^d$, the $d$-dimensional unit ball and $f(t)=t^{-s}$ with  $d-2\leq s<d$, then Theorem \ref{thsimanek} applies, while if $0<s<d-2$ or $f(t)=\log(2/t)$, then the assumptions of this theorem are not satisfied. However, it was shown by Erd\'elyi and Saff \cite{Erdelyi2013} that for this case the only $N$-point normalized counting measure $\nu_N$ that attains $\PP_f(\mathbb{B}^d, N)$ is $\nu_N=\delta_0$.\\

In this paper we obtain  a convergence theorem that holds for all integrable Riesz kernels provided the set $A$ is $d$-regular.

\begin{defin}
A compact set $A\subset \R^p$ is called {\it $d$-regular}, $0<d\leqslant p$, if there exist two positive constants $c$ and $C$ such that for any point $y\in A$ and any $r$ with $0<r<\diam(A)$, we have $cr^d\leqslant \H_d(B(y,r)\cap A)\leqslant Cr^d$, where $\H_d$ is the $d$-dimensional Hausdorff measure on $\R^p$ normalized by $\H_d([0,1]^d)=1$.
\end{defin}
Further, we introduce a	special family of kernels.
\begin{defin}
A function $f\colon (0,\infty)\to (0, \infty)$ is called {\it $d$-Riesz-like} if it is continuous, strictly decreasing, and for some $\ep$ with $0<\ep<d$ and $t_\ep>0$ the function $t\mapsto t^{d-\ep}f(t)$ is increasing on $[0, t_\ep]$; the value at zero is formally defined by
$$
\lim_{t\to 0^+} t^{d-\ep}f(t).
$$
The kernel $K$ is called {$d$-Riesz-like} if $K(x,y)=f(|x-y|)$.
\end{defin}
\begin{zamech*}
Examples of such functions $f$ include $s$-Riesz potentials $f(t)=t^{-s}$ for $0<s<d$, as well as $f(t)=\log (c/t)$, where the constant $c$ is chosen so that $\log(c/|x-y|) >0$ for any $x,y \in A$. Further, we can consider $f(t):=t^{-s} \cdot (\log(c/t))^\alpha$ for any $\alpha>0$ and $0<s<d$. We also do not exclude the case when $f$ is bounded; e.g., $f(t)=e^{-ct^2}$, $c>0$.
\end{zamech*}
Under above assumptions on $A$ and $f$, we first study the behavior of $P_K(\mu_N)$ as $\mu_N\stackrel{*}{\to}\mu$. In what follows, when $K(x,y)=f(|x-y|)$ we write $U_f$, $P_f$ and $T_f(A)$ instead of $U_K$, $P_K$ and $T_K(A)$. We prove the following.
\begin{theorem}\label{maintheorem}
Let $A$ be a $d$-regular compact set, and $f$ be a $d$-Riesz-like function. If $(\nu_N)_{N=1}^\infty$ is a sequence of measures on $A$ with $\nu_N\stackrel{*}{\to} \nu$, then $P_f(\nu_N)\to P_f(\nu)$ as $N\to \infty$.
\end{theorem}
This theorem is a direct consequence of a minimum principle for potentials, introduced below in Theorem \ref{minimumpr}. From Theorem \ref{maintheorem} we derive the following result.
\begin{theorem}\label{weaklimitpolar}
Let $A$ and $f$ satisfy the conditions of Theorem \ref{maintheorem}. For each $N$ let $\nu_N$ be an $N$-point normalized counting measure that attains $\PP_f(A, N)$.
If $\nu^*$ is any weak$^*$-limit measure of the sequence $(\nu_N)$, then $\nu^*$ solves the continuous $f$-polarization problem.
\end{theorem}

Notice that whenever there is a unique measure $\nu$ that solves the continuous polarization problem on $A$,  then Theorem \ref{weaklimitpolar} implies that the whole sequence $\{\nu_N\}$ converges to $\nu$ in the weak$^*$ sense.

\section{A minimum principle for Riesz-like potentials}
We begin this section with some known results from potential theory.
In what follows, all measures will have support on $A$. We proceed with the following definition, important in potential theory.

\begin{defin}
A set $E\subset A$ is called {\it $K$-negligible} if for any compact set $E_1\subset E$ and any measure $\mu$ such that $U_K^\mu$ is bounded on $E_1$, we have $\mu(E_1)=0$.

\end{defin}

The following definition describes a useful class of kernels.
\begin{defin}
The kernel $K$ is said to be {\it regular} if for any positive Borel measure $\mu$ the following is satisfied: if the potential $U_K^\mu$ is finite and continuous on $\supp\mu$, then it is finite and continuous in the whole space $\R^p$.
\end{defin}
It is known, see \cite{Fuglede1960}, that a kernel of the form $K(x,y)=f(|x-y|)$, where $f$ is a continuous non-negative strictly decreasing function, is regular.
Regularity of a kernel implies the following two results.
\begin{theorem}[Principle of descent, Lemma 2.2.1 in \cite{Fuglede1960}]\label{thdescent}
Assume $K$ is regular. If $\mu_n \stackrel{*}{\to} \mu$ and $y_n\to y_\infty$ as $n\to \infty$, then
$$
\liminf_n U_K^{\mu_n}(y_n) \geqslant U_K^\mu(y_\infty).
$$
\end{theorem}
\begin{theorem}[Lower envelope, Theorem 3 in \cite{Brelot1960}]\label{thenvelope}
Assume $K$ is regular. If $\mu_n \stackrel{*}{\to} \mu$, then the set
$$
E:=\{y\in A\colon \liminf_n U_K^{\mu_n}(y) > U_K^{\mu}(y)\},
$$
is $K$-negligible.
\end{theorem}

The new minimum principle mentioned in the title is the following.
\begin{theorem}\label{minimumpr}
Let $A$ be a $d$-regular set, and $f$ be a $d$-Riesz-like function on $(0, \infty)$. If for a measure $\mu$ on $A$ and a constant $M$,
\begin{equation}\label{biggeroutsidenegl}
U_f^\mu(y)\geqslant M, \;\; y\in A\setminus E,
\end{equation}
where $E$ is $f$-negligible, then $U_f^{\mu}(y)\geqslant M$ for every $y\in A$.
\end{theorem}

We proceed with a proposition that is an analog of the Lebesgue differentiation theorem for potentials.
\begin{defin}
For a function $\vf\colon A\to \R$ and a point $y\in A$, we call $y$ a {\it weak $d$-Lebesgue point} of $\vf$ if
$$
\vf(y)=\lim_{r\to 0^+} \frac{1}{\H_d(A\cap B(y,r))} \ili_{A\cap B(y,r)} \vf(z)\textup{d}\H_d(z),
$$
where $B(y,r)$ denotes the open ball in $\R^p$ with center at $y$ and radius $r$.
\end{defin}
\begin{prop}\label{convolv}
Suppose $A$ and $f$ satisfy conditions of Theorem \ref{minimumpr}, and $\mu$ is a measure supported on $A$. Then every point $y\in A$ is a weak $d$-Lebesgue point of $U_f^{\mu}$.
\end{prop}
We start by proving the following technical lemma.
\begin{lemma}\label{sublemma}
There exist positive numbers $C_0$ and $r_0$ such that for any $x\in A$ and any $r<r_0$\textup:
\begin{equation}\label{technicalineq}
\frac{1}{\H_d(A\cap B(y,r))} \cdot \frac{1}{f(|x-y|)} \cdot \int_{A\cap B(y,r)} f(|x-z|)\textup{d}\H_d(z) \leqslant C_0.
\end{equation}
\end{lemma}
\begin{proof}
Notice that the left-hand side of \eqref{technicalineq} is equal to
$$
\frac{1}{\H_d(A\cap B(y,r))} \cdot \frac{1}{f(|x-y|)} \cdot \int_0^\infty \H_d(z\in A\cap B(y,r)\colon f(|x-z|)>u) \textup{d}u.
$$
Since $f$ is decreasing, we see that
$$
\{z\in A\cap B(y,r)\colon f(|x-z|)>u\} = \{z\in A\cap B(y,r)\cap B(x, f^{-1}(u))\}.
$$
This set is empty when $f^{-1}(u)<|x-y|-r$ or $u>f(|x-y|-r)$.

We consider two cases.
\paragraph{\textbf{Case 1}: $|x-y|>2r$}

Then we obtain the estimate
\begin{align}
\begin{split}
&\frac{1}{\H_d(A\cap B(y,r))} \cdot \frac{1}{f(|x-y|)} \cdot \int_0^\infty \H_d(z\in A\cap B(y,r)\colon f(|x-z|)>u) \textup{d}u  \\
&=\frac{1}{\H_d(A\cap B(y,r))} \cdot \frac{1}{f(|x-y|)} \cdot \int_0^{f(|x-y|-r)} \H_d(z\in A\cap B(y,r)\colon f(|x-z|)>u) \textup{d}u  \\
&\leqslant \frac{f(|x-y|-r)}{f(|x-y|)}.
\end{split}
\end{align}
Since $|x-y|>2r$, we have $|x-y|-r \geqslant |x-y|/2$; thus,
$$
\frac{1}{\H_d(A\cap B(y,r))} \cdot \frac{1}{f(|x-y|)} \cdot \int_0^\infty \H_d(z\in A\cap B(y,r)\colon f(|x-z|)>u) \textup{d}u \leqslant \frac{f(|x-y|/2)}{f(|x-y|)}.
$$
Finally, if recall that the function $t\mapsto f(t)\cdot t^{d-\ep}$ is increasing for $t\in [0, t_\ep]$. If $|x-y|>t_\ep$, we get
$$
\frac{f(|x-y|/2)}{f(|x-y|)} \leqslant \frac{f(t_\ep/2)}{f(\textup{diam}(A))}.
$$
If $|x-y|\leqslant t_\ep$, we use that
$$
f(|x-y|/2) \cdot (|x-y|/2)^{d-\ep} \leqslant f(|x-y|) \cdot (|x-y|)^{d-\ep};
$$
thus
$$
\frac{f(|x-y|/2)}{f(|x-y|)} \leqslant 2^{d-\ep}.
$$

\paragraph{\textbf{Case 2}: $|x-y|\leqslant 2r$}
Again, we need only integrate for $u\leqslant f(|x-y|-r)$. Setting $f$ equal to $f(0)$ for any negative argument, we write
\begin{multline}
\frac{1}{\H_d(A\cap B(y,r))} \cdot \frac{1}{f(|x-y|)} \cdot \int_0^{f(|x-y|-r)} \H_d(z\in A\cap B(y,r)\colon f(|x-z|)>u) \textup{d}u = \\
\frac{1}{\H_d(A\cap B(y,r))} \cdot \frac{1}{f(|x-y|)} \cdot \left( \int_0^{f(|x-y|+r)} + \int_{f(|x-y|+r)}^{f(|x-y|-r)} \H_d(z\in A\cap B(y,r)\colon f(|x-z|)>u) \textup{d}u\right)
\end{multline}
Trivially,
\begin{align}
\begin{split}
\frac{1}{\H_d(A\cap B(y,r))} \cdot &\frac{1}{f(|x-y|)} \cdot  \int_0^{f(|x-y|+r)}\H_d(z\in A\cap B(y,r)\colon f(|x-z|)>u) \textup{d}u \\ &\leqslant\frac{f(|x-y|+r)}{f(|x-y|)} \leqslant 1.
\end{split}
\end{align}
Furthermore, since $|x-y|\leqslant 2r$, we have $f(|x-y|)\geqslant f(2r)$; thus,
\begin{align}
\begin{split}
\frac{1}{\H_d(A\cap B(y,r))} \cdot &\frac{1}{f(|x-y|)} \int_{f(|x-y|+r)}^{f(|x-y|-r)} \H_d(z\in A\cap B(y,r)\colon f(|x-z|)>u) \textup{d}u \\
&\leqslant
\frac{1}{c\cdot r^d} \cdot \frac{1}{f(2r)}\int_{f(|x-y|+r)}^{\infty} \H_d(z\in A\cap B(x, f^{-1}(u)) \textup{d}u  \\
&\leqslant \frac{C/c}{r^d f(2r)} \int_{f(|x-y|+r)}^{\infty} (f^{-1}(u))^d \textup{d}u.
\end{split}
\end{align}
Note that the assumption $|x-y|\leqslant 2r$ implies $f(|x-y|+r)\geqslant f(3r)$, and thus
\begin{align}
\begin{split}
\frac{1}{\H_d(A\cap B(y,r))} \cdot &\frac{1}{f(|x-y|)}\int_{f(|x-y|+r)}^{f(|x-y|-r)} \H_d(z\in A\cap B(y,r)\colon f(|x-z|)>u) \textup{d}u \\ &\leqslant\frac{C/c}{r^d f(2r)} \int_{f(3r)}^{\infty} (f^{-1}(u))^d \textup{d}u.
\end{split}
\end{align}
We now observe that our assumption that $f(t)t^{d-\ep}$ is increasing on $[0, t_\ep]$ implies that the function $u^{1/(d-\ep)}f^{-1}(u)$ is decreasing on $[f(t_\ep), \infty)$. Therefore, for $3r<t_\ep$ we have
\begin{multline}\label{inequalityeight}
\frac{C/c}{r^d f(2r)} \int_{f(3r)}^{\infty} (f^{-1}(u))^d \textup{d}u = \frac{C/c}{r^d f(2r)} \int_{f(3r)}^{\infty} (u^{1/(d-\ep)}f^{-1}(u))^d \cdot u^{-d/(d-\ep)} \textup{d}u  \\
\leqslant \frac{C_1}{r^d f(2r)} \cdot f(3r)^{d/(d-\ep)} \cdot (3r)^d \cdot f(3r)^{1-d/(d-\ep)} \leqslant C_2.
\end{multline}

\end{proof}
\begin{proof}[Proof of Proposition \ref{convolv}]
We formally define $f(0):=\lim_{t\to 0^+} f(t)\in (0, \infty]$. Without loss of generality, we consider the case $f(0)=\infty$; otherwise the potential $U_f^\mu$ is continuous on $\R^p$ and the proposition holds trivially.

Define
$$
\Phi^\mu_r(y):= \frac{1}{\H_d(A\cap B(y,r))} \int_{A\cap B(y,r)} U_f^\mu(z)\textup{d}\H_d(z).
$$
Tonneli's theorem and Lemma \ref{sublemma} imply
\begin{align}\label{tonneli}
\begin{split}
&\Phi^\mu_r(y)= \frac{1}{\H_d(A\cap B(y,r))} \int_{A\cap B(y,r)} \int_A f(|x-z|) \textup{d}\mu (x) \textup{d}\H_d(z) \\
&=\int_A f(|x-y|) \cdot \frac{1}{\H_d(A\cap B(y,r))}\cdot \frac{1}{f(|x-y|)} \int_{A\cap B(y,r)}f(|x-z|)\textup{d}\H_d(z)\textup{d}\mu(x) \\ &\leqslant C_0 \int_A f(|x-y|) \textup{d}\mu(x) = C_0 U_f^{\mu}(y).
\end{split}
\end{align}

We first suppose $U^\mu(y)=\infty$. Since $U^\mu$ is lower semi-continuous, we obtain that for any large number $N$ there is a positive number $r_N$, such that $U^\mu(x)>N$ in $B(y, r_N)$. Then for any $r<r_N$ we get
$$
\Phi^\mu_r(y)=\frac{1}{\H_d(A\cap B(y,r))} \int_{A\cap B(y,r)}U_f^\mu(z)\textup{d}\H_d(z) >N.
$$
This implies that $\Phi^\mu_r(y)\to \infty = U_f^\mu(y)$ as $r\to 0^+$.

\bigskip

Now assume $U_f^\mu(y) = \int_A f(|x-y|)\textup{d}\mu(x)<\infty$. Notice that since $f(0)=\infty$, the measure $\mu$ cannot have a mass point at $y$. Consequently, for any $\eta>0$ there exists a ball $B(y,\delta)$, such that
$$
\int_{B(y,\delta)} f(|x-y|)\textup{d}\mu(x) < \eta.
$$
Consider measures
$$
\textup{d}\mu':=\mathbbm{1}_{B(y,\delta)} \textup{d}\mu, \; \; \mbox{and} \; \; \mu_c:=\mu - \mu'.
$$
Since $y\not\in\textup{supp}(\mu_c)$, the potential $U_f^{\mu_c}$ is continuous at $y$. This implies
$$
\Phi^{\mu_c}_r(y) = \frac{1}{\H_d(A\cap B(y,r))} \int_{A\cap B(y,r)}U_f^{\mu_c}(z)\textup{d}\H_d(z)\to U_f^{\mu_c}(y), \; \; r\to 0^+.
$$
Also, on applying \eqref{tonneli} to $\mu'$, it follows that
$$
\Phi^\mu_r(y) = \Phi^{\mu_c}_r(y) + \Phi^{\mu'}_r(y) \leqslant \Phi^{\mu_c}_r(y) + C_0 U_f^{\mu'}(y) \leqslant \Phi^{\mu_c}_r(y)+C_0 \eta.
$$
Taking the $\limsup$, we obtain
\begin{equation}\label{kvaks1}
\limsup_{r\to 0^+}\Phi^{\mu}_r(y) \leqslant U_f^{\mu_c}(y)+C_0\eta \leqslant U_f^{\mu}(y)+C_0\eta.
\end{equation}
On the other hand, we know that
$$
\Phi^{\mu}_r(y) = \Phi^{\mu_c}_r(y) + \Phi^{\mu'}_r(y)\geqslant \Phi^{\mu_c}_r(y),
$$
and thus
$$
\liminf_{r\to 0^+} \Phi^{\mu}_r(y)  \geqslant U_f^{\mu_c}(y) = U_f^{\mu}(y) - U_f^{\mu'}(y) \geqslant U_f^{\mu}(y)-\eta.
$$
This, together with \eqref{kvaks1} and the arbitrariness of $\eta$ implies the assertion of Proposition \ref{convolv}.
\end{proof}
%
We are ready to deduce Theorem \ref{minimumpr}.
\begin{proof}[Proof of Theorem \ref{minimumpr}]
We first claim that any $f$-negligible subset $E$ of $A$ has $\H_d$-measure zero. Indeed, take any compact set $E_1$ inside our $f$-negligible set $E$. Then for any $y\in A$
$$
U_f^{\H_d}(y)= \int_A f(|x-y|)\textup{d}\H_d(x) = \int_{0}^{\infty} \H_d(x\in A\cap B(y, f^{-1}(u)))\textup{d}u.
$$
We notice that if $u<f(\textup{diam}(A))=:u_0$, then $f^{-1}(u)>\textup{diam}(A)$, and so $A\cap B(y, f^{-1}(u)) = A$. Thus,
\begin{align}\label{kvakskoaks}
\begin{split}
U_f^{\H_d}(y) &= \left(\int_{0}^{u_0} + \int_{u_0}^\infty\right) \H_d(x\in A\cap B(y, f^{-1}(u)))\textup{d}u \\
&\leqslant u_0 \H_d(A) + C \int_{u_0}^\infty (f^{-1}(u))^d \textup{d}u,
\end{split}
\end{align}
which is bounded by a constant that does not depend on $y$, as proved in the Case 2 of Lemma \ref{sublemma}; see inequality \eqref{inequalityeight}. Since the set $E$ is negligible, we conclude that $\H_d(E_1)=0$. Thus, for any compact subset $E_1$ of $E$ we have $\H_d(E_1)=0$ and so $\H_d(E)=0$ as claimed.
Now, let the measure $\mu$ satisfy \eqref{biggeroutsidenegl}. Then for any $y\in A$, we deduce from Proposition \ref{convolv} and the fact that
$
U_f^{\mu}(z)\geqslant M
$
holds $\H_d$-a.e. on $A$ that
$
U_f^\mu(y)\geqslant M
$.
\end{proof}
\section{Proof of Theorem \ref{maintheorem} and Theorem \ref{weaklimitpolar}}

\begin{proof}[Proof of Theorem \ref{maintheorem}]
For any increasing infinite subsequence $\mathcal{N}\subset \mathbb{N}$, choose a subsequence $\mathcal{N}_1$, such that
$$
\liminf_{N\in \mathcal{N}} P_f(\nu_N)=\lim_{N\in \mathcal{N}_1} P_f(\nu_N).
$$
For each $N\in \mathcal{N}_1$ take a point $y_N$, such that $P_f(\nu_N)=U^{\nu_N}(y_N)$. Passing to a further subsequence $\mathcal{N}_0\subset \mathcal{N}_1$ we can assume $y_N\to y_\infty$ as $N\to \infty$, $N\in \mathcal{N}_0$. Then the principle of descent, Theorem \ref{thdescent}, implies
\begin{equation}\label{tralyalyaleq}
\liminf_{N\in \mathcal{N}} P_f(\nu_N) = \lim_{N\in \mathcal{N}_0} U_f^{\nu_N}(y_N) \geqslant U_f^\nu(y_\infty) \geqslant P_f(\nu).
\end{equation}
Furthermore, for any $y\in A$ we have
$$
\liminf_{N\in \mathcal{N}} U_f^{\nu_N}(y) \geqslant \liminf_{N\in \mathcal{N}} P_f(\nu_N)=: M_f(\mathcal{N}), \; \; y\in A.
$$
By Theorem \ref{thenvelope}, $\liminf_{N\in \mathcal{N}} U_f^{\nu_N}(y) = U_f^{\nu}(y)$ for every $y\in A\setminus E$, where $E$ is an $f$-negligible set that can depend on $\mathcal{N}$. Therefore,
$$
U_f^{\nu}(y)\geqslant M_f(\mathcal{N}), \; \; y\in A\setminus E.
$$
From the minimum principle, Theorem \ref{minimumpr}, we deduce that
$$
U_f^{\nu}(y)\geqslant M_f(\mathcal{N}) =\liminf_{N\in \mathcal{N}} P_f(\nu_N), \; \; \forall y\in A,
$$
and therefore
\begin{equation}\label{tralyalyageq}
P_f(\nu)\geqslant \liminf_{N\in \mathcal{N}} P_f(\nu_N).
\end{equation}
Combining estimates \eqref{tralyalyaleq} and \eqref{tralyalyageq}, we deduce that for any subsequence $\mathcal{N}$ we have
$$
\liminf_{N\in \mathcal{N}} P_f(\nu_N) = P_f(\nu).
$$
This immediately implies
$$
\lim_N P_f(\nu_N)=P_f(\nu),
$$
which completes the proof.
\end{proof}

\begin{proof}[Proof of Theorem \ref{weaklimitpolar}]
Assume for a subsequence $\mathcal{N}$ we have $\nu_N \stackrel{*}{\to} \nu^*$ as $N\to \infty$, $N\in \mathcal{N}$.
From \cite{Ohtsuka1967} we know that
$$
P_f(\nu_N)\to T_f(A), \;\; N\to \infty.
$$
On the other hand, from Theorem \ref{maintheorem} we know that
$$
P_f(\nu_N)\to P_f(\nu^*), \;\; N\to \infty, \;\; N\in \mathcal{N}.
$$
Therefore,
$$
P_f(\nu^*)= T_f(A),
$$
which proves the theorem.
\end{proof}
\begin{zamech*}
{\color{black}Suppose $A=\cup_{k=1}^m A_k$, where $A_k$ is a $d_k$-regular compact set, and that,  for some positive number $\delta$,  we have $\textup{dist}(A_{i}, A_{j})\geqslant \delta$ for  $i\not=j$. Further assume that $f$ is a $d_k$-Riesz-like kernel for every $k=1,\ldots, m$, and $\mu$ is a measure supported on $A$. Then the result of Theorem \ref{minimumpr}, and thus of Theorems \ref{maintheorem} and \ref{weaklimitpolar} hold. To see this, we first show that every $y\in A_k$ is a weak $d_k$-Lebesgue point of $U_f^\mu$.
Indeed, setting $\textup{d}\mu_k := \mathbbm{1}_{A_k} \textup{d}\mu$ yields
$$
U_f^\mu(y)=\sli_{k=1}^m U_f^{\mu_k}(y).
$$
Proposition \ref{convolv} implies that if $y\in A_k$, then $y$ is a weak $d_k$-Lebesgue point of $U_f^{\mu_k}$. Moreover, for any $j\not=k$ we have $y\not\in \supp(\mu_j)$; thus $U_f^{\mu_j}$ is continuous at $y$, and our assertion about weak Lebesgue points of $U_f^{\mu}$ follows.
Similar to \eqref{kvakskoaks}, we then deduce that if a set $E\subset A$ is $f$-negligible, then $\H_{d_k}(E\cap A_k)=0$ for every $k=1,\ldots, m$; therefore, the assertion of Theorem \ref{minimumpr} remains true and the proofs of Theorems \ref{maintheorem} and \ref{weaklimitpolar} go exactly as before. }
\end{zamech*}

\bibliography{ref}

\bibliographystyle{plain}
\end{document}